\documentclass{amsart}

\usepackage{amsmath}
\usepackage{amssymb}
\usepackage{amsthm}
\usepackage{tikz}
\usetikzlibrary{matrix,arrows,backgrounds,shapes.misc,shapes.geometric,patterns,calc,positioning}

\usepackage{ifthen}

\newcommand{\ZZ}{\mathbb{Z}}

\newcommand{\Aut}{\operatorname{Aut}}  

\newcommand{\id}{\operatorname{id}}

\newcommand{\e}{\epsilon}

\newcommand{\Gr}{\operatorname{Gr}}

\newcommand{\GL}{\operatorname{GL}}

\newcommand{\AlgGr}{\operatorname{AlgGr}} 
 
\newcommand{\msa}{\mathfrak{msa}}

\newcommand{\Iso}{\operatorname{Iso}} 
\newcommand{\Inn}{\operatorname{Inn}}

\newtheorem{theorem}{Theorem}[section]
\newtheorem{proposition}[theorem]{Proposition}
\newtheorem{lemma}[theorem]{Lemma}
\newtheorem{corollary}[theorem]{Corollary}

\newtheorem{example}[theorem]{Example}

\theoremstyle{remark}

\renewenvironment{proof}{{\noindent\bf Proof.}}{\hfill $\Box$\par\vskip3mm} 
\newenvironment{proofEven}{{\noindent\bf Proof (Even Case).}}{\hfill $\Box$\par\vskip3mm}  
\newenvironment{proofOdd}{{\noindent\bf Proof (Odd Case).}}{\hfill $\Box$\par\vskip3mm}

\begin{document}    

\title[Isoclasses Determined by Automorphisms]{On Isoclasses of Maximal Subalgebras Determined by Automorphisms} 
\author{Alex Sistko} 
\address{Department of Mathematics, University of Iowa, 
Iowa City, IA 52242}
\email{alexander-sistko@uiowa.edu}

\maketitle{}
\date{}

\begin{abstract}

 Let $k$ be an algebraically-closed field, and let $B = kQ/I$ be a basic, finite-dimensional associative $k$-algebra with $n := \dim_kB < \infty$. Previous work shows that the collection of maximal subalgebras of $B$ carries the structure of a projective variety, denoted by $\msa (Q)$, which only depends on the underlying quiver $Q$ of $B$. The automorphism group $\Aut_k(B)$ acts regularly on $\msa (Q)$. Since $\msa (Q)$ does not depend on the admissible ideal $I$, it is not necessarily easy to tell when two points of $\msa (Q)$ actually correspond to isomorphic subalgebras of $B$. One way to gain insight into this problem is to study $\Aut_k(B)$-orbits of $\msa (Q)$, and attempt to understand how isoclasses of maximal subalgebras decompose as unions of $\Aut_k(B)$-orbits. This paper investigates the problem for $B = kQ$, where $Q$ is a type $\mathbb{A}$ Dynkin quiver. We show that for such $B$, two maximal subalgebras with connected Ext quivers are isomorphic if and only if they lie in the same $\Aut_k(B)$-orbit of $\msa (Q)$.  
\end{abstract}  


\section{Introduction}\label{s.intro}  

\noindent Let $B$ be a finite-dimensional, unital, associative algebra over an algebraically-closed field $k$. Then the celebrated \emph{Wedderburn-Malcev Theorem} states that there exists a $k$-subalgebra $B_0 \subset B$ such that $B_0 \cong B/J(B)$ and $B = B_0 \oplus J(B)$, where $J(B)$ denotes the Jacobson radical of $B$. Furthermore, for any subalgebra $B_0' \subset B$ isomorphic to $B_0$, there exists a $x \in J(B)$ such that $(1+x)B_0(1+x)^{-1} = B_0'$. For more details, see for instance \cite{Farn} or Theorem 11.6 of \cite{Pi}. Of course, the collection of all $k$-algebra automorphisms of $B$, which we denote by $\Aut_k(B)$, acts on the set of subalgebras of $B$. For any $x \in J(B)$, the map $y \mapsto (1+x)y(1+x)^{-1}$ is an automorphism of $B$. So another way to state the second half of the Wedderburn-Malcev Theorem is to say that the \emph{isoclass} of $B_0$ in $B$, i.e. the set of all subalgebras of $B$ isomorphic to $B_0$, is a single $\Aut_k(B)$-orbit. 

Unsurprisingly, this statement is false for general subalgebras $A$ of $B$. Nevertheless, recent investigations into maximal subalgebras of finite-dimensional algebras suggest that examples of such $A$ are not necessarily rare \cite{IS}. It is therefore natural to ask what conditions we can impose on $A$ to ensure that its isoclass in $B$ is an orbit of $\Aut_k(B)$. More generally, one can ask whether there is any way to classify the $\Aut_k(B)$-orbits of subalgebras of $B$, and relate them to isoclasses of subalgebras. This is one source of inspiration for the current paper.

Another source of inspiration comes from the study of varieties of subalgebras, as the author has done recently in \cite{Sis}. For any $1 \le m \le \dim_kB$, the collection of all $m$-dimensional subalgebras of $B$ carries the structure of a projective $k$-variety, which we call $\AlgGr_m(B)$. The linear algebraic group $\Aut_k(B)$ acts regularly on this variety. Neither $m$ or $B$ are enough to specify $\AlgGr_m(B)$ up to equivalence of varieties: in fact, if $B = kQ/I$ is a basic algebra and $m = \dim_kB - 1$, then $\AlgGr_m(B)$ \emph{only} depends on $Q$. So it will be difficult in general to choose an admissible ideal $I$ of $kQ$, and determine whether two points of $\AlgGr_{\dim_kkQ/I-1}(kQ/I)$ actually represent isomorphic subalgebras of $kQ/I$. Thankfully, the automorphism group $\Aut_k(kQ/I)$ \emph{is} sensitive to the data contained in $I$. So, provided that one can impose reasonable conditions on the relationship between orbits and isoclasses, one can expect that orbits under this group action will yield significant information on isoclasses of subalgebras. In \cite{Sis} we discuss one possible version of ``reasonable conditions,'' where the variety is a finite union of orbits.

The purpose of this paper is to carry out this program as far as possible for a suitable ``test class'' of algebras. For us, these will be path algebras of type $\mathbb{A}$ Dynkin quivers and their maximal subalgebras. As it turns out, many maximal subalgebras of such algebras will have isoclasses that are single $\Aut_k(B)$-orbits. However, we will show that even for such a nicely-behaved class, isoclasses differ from orbits in at least some circumstances.

This paper is organized as follows. In Section \ref{s.bg} we review the basic notions associated to path algebras and their automorphisms. We also discuss the major results from \cite{IS}, \cite{Sis} which will be used to prove our main result. In Section \ref{s.pres}, we discuss the problem of presenting maximal subalgebras of basic algebras. In particular, Propositions \ref{p.seppres} and \ref{c.pathalg} provide explicit presentations for maximal subalgebras of hereditary algebras. The results of this section will be used in Section \ref{s.main}, where we prove the main result of this article: 

\begin{theorem}\label{t.main}
Let $k$ be an algebraically closed field, $Q$ a type $\mathbb{A}$ Dynkin quiver, and $B = kQ$. Suppose that $A, A' \in \msa (Q)$ have connected Ext quivers. Then $A$ and $A'$ lie in the same $\Aut_k(B)$-orbit if and only if $A \cong A'$ as $k$-algebras.
\end{theorem} 

\noindent This is essentially done by showing that the Ext quivers of $A$ and $A'$ are nonisomorphic whenever they lie in different $\Aut_k(B)$-orbits. We note that this theorem can be rephrased as follows: if the underlying graph of $Q$ is an oriented tree with maximum degree $2$ and $B = kQ$, then the isoclass of a connected maximal subalgebra of $B$ coincides with its $\Aut_k(B)$-orbit. The author does not currently know whether similar statements hold for all trees with maximum degree $3$ or higher. We end on an example which shows that if the Ext quiver of $A$ is not connected, then its isoclass can differ from its $\Aut_k(B)$-orbit.

\section{Background}\label{s.bg}  

Unless otherwise stated, $k$ will denote an algebraically-closed field. All algebras are unital, associative, finite-dimensional $k$-algebras, and our terminology essentially comes from \cite{ASS}. Let $Q$ be a finite quiver with vertex set $Q_0$, arrow set $Q_1$, and source (resp. target) function $s$ (resp. $t$) $: Q_1\rightarrow Q_0$. The \emph{underlying graph} of $Q$ is obtained by forgetting the orientations on the arrows. Let $kQ$ denote the path algebra of $Q$, and let $J(Q)$ denote the two-sided ideal in $kQ$ generated by $Q_1$. For $n \geq 2$, we let $T_nQ := kQ/J(Q)^n$ denote the \emph{$n^{th}$ truncated path algebra} associated to $Q$. By a slight abuse of notation, for any $u, v \in Q_0$ we let $uQ_1v$ denote the set of arrows in $Q$ with source $u$ and target $v$, and we let $ukQ_1v$ denote their $k$-span inside $kQ$. Note that if $uQ_1v = \emptyset$, then $ukQ_1v = \{ 0 \}$ and $\GL(ukQ_1v)$ is the trivial group. Similar to \cite{GS2}, we define $V^2(Q) = \{ (u,v) \in Q_0\times Q_0\mid uQ_1v \neq \emptyset\}$. A \emph{basic algebra} is an algebra of the form $B = kQ/I$, where $I$ is an \emph{admissible ideal} of $kQ$, i.e. an ideal satisfying $J(Q)^2 \supset I \supset J(Q)^{\ell}$ for some $\ell \geq 2$. Note that $B = kQ_0 \oplus J(B) = kQ_0 \oplus J(Q)/I$, and that $kQ_0 \cong B/J(B) \cong k^{|Q_0|}$. 


We let $\Aut_k(B)$ denote the group of all $k$-algebra automorphisms of $B$. It is a Zariski-closed subgroup of $\GL(B)$, and hence a linear affine algebraic group. Our notation for subgroups of $\Aut_k(B)$ is borrowed from the notation in \cite{Poll}, \cite{GS1}, \cite{GS2}. If $G$ is a subgroup of $\Aut_k(B)$, we say that two subalgebras $A$ and $A'$ are \emph{$G$-conjugate} (in $B$) if there exists a $\phi \in G$ such that $A' = \phi (A)$. For a unit $u \in B^{\times}$, we let $\iota_u$ denote the corresponding \emph{inner automorphism}, i.e. the map $\iota_u(x) = uxu^{-1}$ for all $x \in B$. We let $\Inn (B)$ denote the group of all inner automorphisms, and $\Inn^*(B) = \{ \iota_{1+x} \mid x \in J(B)  \}$ denote the group of \emph{unipotent inner automorphisms}. If $B = kQ/I$ is basic, we let $\hat{H}_B = \{ \phi \in \Aut_k(B) \mid \phi (Q_0) = Q_0 \}$ and $H_B = \{ \phi \in \Aut_k(B) \mid \phi\mid_{Q_0} = \id_{Q_0} \}$. By Theorem 10.3.6 of \cite{HGK}, $\Inn(B)$ acts transitively on complete collections of primitive orthogonal idempotents. Since inner automorphisms induced by units of the form $\sum_{v \in Q_0}{\lambda_vv}$ (where $\lambda_v \in k^{\times}$ for each $v$) fix vertices, $\Inn^*(B)$ is also transitive on this set and we have a decomposition $\Aut_k(B) = \Inn^*(B) \cdot \hat{H}_B = \hat{H}_B\cdot \Inn^*(B)$. If the underlying graph of $Q$ is a tree, then $\Aut (Q)$ can be considered a subgroup of $\Aut_k(T_nQ)$ for any $n$, and it is easy to see that we have a decomposition $\hat{H}_{T_nQ} = \Aut (Q) \cdot H_{T_nQ} = H_{T_nQ} \cdot \Aut (Q)$. Since $J(Q)^n = 0$ for large powers $n$, this statement includes the fact that $\Aut(Q)$ is a subgroup of $\Aut_k(kQ)$.



 In Theorem 4.1 of \cite{IS}, the author and M. C. Iovanov proved the following classification Theorem for maximal subalgebras of basic algebras: 

\begin{theorem}\label{t.basic}
Let $B = kQ/I$ be a basic algebra over an algebraically-closed field $k$. Let $A \subset B$ be a maximal subalgebra. Consider the following two classes of maximal subalgebras of $B$: 
\item For a two-element subset $\{ u,v \} \subset Q_0$, we define  
\begin{equation*}
A(u+v) := k(u+v)\oplus \left( \bigoplus_{w \in Q_0\setminus\{ u,v\}}{kw} \right) \oplus J(B). 
\end{equation*}
\item For an element $(u,v) \in V^2(Q)$ and a codimension-$1$ subspace $U \le ukQ_1v$, we define 
\begin{equation*} 
A(u,v,U) := kQ_0 \oplus U \oplus \left( \bigoplus_{(w,y) \in Q_0^2\setminus\{(u,v)\}}{wkQ_1y} \right) \oplus J(B)^2. 
\end{equation*}  

\noindent Then there exists a unipotent inner automorphism $\iota_{1+x} \in \Inn^*(B)$ such that either $\iota_{1+x}(A) = A(u+v)$ or $\iota_{1+x}(A) = A(u,v,U)$, for some appropriate choice of $u$, $v$, and possibly $U$.
\end{theorem}  

\noindent As in \cite{IS}, if $A$ is $\Inn^*(B)$-conjugate to a subalgebra of the form $A(u+v)$, then we say that $A$ is of \emph{separable type}. If $A$ is $\Inn^*(B)$-conjugate to an algebra of the form $A(u,v,U)$, then we say that $A$ is of \emph{split type}. As an immediate consequence of Theorem \ref{t.basic}, all maximal subalgebras of a basic algebra are basic, have codimension $1$, and contain the radical square. In fact, we have the following easy corollary, which first appeared in \cite{Sis}: 

 \begin{corollary}\label{c.inherit}
 Let $B$ be a basic $k$-algebra of dimension $n$, and let $A \subset B$ be a subalgebra. Then $A$ satisfies the following: 
\begin{enumerate} 
\item $A$ is also a basic algebra.
\item If $A$ is a maximal subalgebra, then $\dim_kA = n-1$.
\item If $A$ is a maximal subalgebra, then $J(A)$ is a $B$-subbimodule of $J(B)$, $J(A) = A \cap J(B)$, and $J(B)^2 \subset J(A)$.
\item More generally, if $m = \dim_kA$, then $J(B)^{2(n-m)} \subset A$. 
\end{enumerate} 
\end{corollary}

If $B$ is any $k$-algebra and $m$ is a positive integer $1 \le m \le \dim_kB$, then the collection $\AlgGr_{m}(B)$ of all $m$-dimensional subalgebras of $B$ is a Zariski-closed subset of the usual Grassmannian $\Gr_m(B)$. In particular, it is a projective variety over $k$. For any $A \in \AlgGr_m(B)$, we let $\Iso (A,B)$ denote the set of all $A' \in \AlgGr_m(B)$ such that $A\cong A'$ as $k$-algebras. Clearly $\Iso (A,B)$ is $\Aut_k(B)$-invariant, and hence a union of $\Aut_k(B)$-orbits. Suppose that $B = kQ/I$ is a basic algebra of dimension $n$. By the remarks above, it follows that $\AlgGr_{n-1}(B)$ is the variety of maximal subalgebras of $B$, and that there is a (biregular) bijection between maximal subalgebras of $B$ and maximal subalgebras of $B/J(B)^2 \cong T_2Q$. In other words, $\AlgGr_{n-1}(B)$ only depends on the underlying quiver $Q$, and so we define $\msa (Q) := \AlgGr_{n-1}(B)$. We can think of $\msa (Q)$ as the variety of maximal subalgebras of \emph{any} basic algebra with Ext quiver $Q$. See \cite{Sis} for more details.

Suppose that the underlying graph of $Q$ is a tree and $B = kQ$. Theorem \ref{t.basic} classifies $\Inn^*(B)$-orbits of $\msa (Q)$, and it is easy to see that for any such $B$, $\phi (A) = A$ for all $\phi \in H_B$. So classification of $\Aut_k(B)$-orbits boils down to determining which $\Inn^*(B)$-orbits of $\msa (Q)$ are related by elements of $\Aut (Q)$. More specifically, it is equivalent to classifying $\Aut (Q)$-orbits on the finite sets $V^2(Q)$ (for split type) and $\{ \{ u,v\} \subset Q_0 \mid u \neq v\}$ (for separable type). Although this may represent an intractable problem for general $Q$, it at least implies that every $B$-isoclass of $\msa (Q)$ is a finite union of $\Aut_k(B)$-orbits. In Section \ref{s.main} we will show that if $Q$ is a type $\mathbb{A}$ Dynkin quiver, then each $B$-isoclass consisting of connected algebras is a \emph{single} $\Aut_k(B)$-orbit. The first step will be to find presentations for each maximal subalgebra as a bound quiver algebra, which we do below. 


\section{Presentations of Maximal Subalgebras}\label{s.pres}

Corollary \ref{c.inherit} (1) states that if $B$ is basic, then all of its maximal subalgebras are also basic. In particular, they can be presented as bound quiver algebras. Ideally, one hopes for explicit presentations of maximal subalgebras in terms of a given presentation for $B$. More specifically, if $B$ is given as $B = kQ/I$ and $A$ is a maximal subalgebra of $B$, one would like a combinatorial procedure to obtain the Ext quiver of $A$, call it $\Gamma$, from $Q$, and another procedure to find generators for the kernel of the projection map $k\Gamma \rightarrow A$. As it currently stands, if $I \subset kQ$ is an arbitrary admissible ideal, and $A \subset kQ/I$ is a maximal subalgebra of split type, then it is not clear to the author how one can explicitly reconstruct $\Gamma$ from $Q$. Nevertheless, Theorem \ref{t.basic} provides us with some insight into the presentation problem. In fact, it is good enough to give us a full description for separable type subalgebras, as well as explicit presentations for \emph{all} maximal subalgebras in the hereditary case, i.e. $I = \{ 0 \}$.    

We start by describing presentations for maximal subalgebras of separable type. Take a $2$-element subset $\{ u,v\} \subset Q_0$, and let $\Gamma$ be the quiver obtained from $Q$ by gluing $u$ and $v$ together. More explicitly, $\Gamma$ has vertex set $Q_0\setminus\{ u,v\} \cup \{u+v\}$, and for all $w,y \in Q_0\setminus\{ u,v\}$ we have 

\begin{equation*} 
(u+v)\Gamma_1y = uQ_1y \cup vQ_1y, 
\end{equation*} 
\begin{equation*} 
w\Gamma_1(u+v) = wQ_1u\cup wQ_1v, 
\end{equation*} 
\begin{equation*} 
(u+v)\Gamma_1(u+v) = uQ_1u \cup uQ_1v \cup vQ_1u \cup vQ_1v. 
\end{equation*} 

\noindent In other words, $\Gamma_1$ is just a re-partitioning of $Q_1$ into arrows with possibly new endpoints. This induces a bijective map $\phi : Q_1 \rightarrow \Gamma_1$. Hence, if $p = \alpha_1\cdots \alpha_d$ is a path in $Q$, then $\phi (p) := \phi (\alpha_1)\cdots \phi(\alpha_d)$ is a well-defined path in $\Gamma$. We can extend this to an algebra map $\phi : kQ \rightarrow k\Gamma$ by defining
\begin{equation*} 
\phi(w) = w \text{ for all $w \in Q_0\setminus\{ u,v\}$,}
\end{equation*} 
\begin{equation*} 
\phi(u) = \phi(v) = u+v,
\end{equation*} 

\noindent and extending to $k$-linear combinations of arbitrary paths.

\begin{proposition}\label{p.seppres}
Let $B = kQ/I$, and $A$ a maximal subalgebra of separable type. Suppose that $A$ is $\Inn^*(B)$-conjugate to $A(u+v)$, for some two-element subset $\{u,v\} \subset Q_0$. Then $A \cong k\Gamma/I'$, where 
\begin{enumerate} 
\item $\Gamma$ is obtained from $Q$ by gluing vertices $u$ and $v$.  
\item $I'$ is generated by relations in $\phi (I)$, along with elements of the form $\phi(\alpha) \phi(\beta)$, where either $\alpha \in Q_1u$ and $\beta \in vQ_1$, or $\alpha \in Q_1v$ and $\beta \in uQ_1$.
\end{enumerate}
\end{proposition}    

\begin{proof} 
There is a map $ k\Gamma \rightarrow A(u+v)$ which acts as the identity on $Q_0\setminus \{u,v\}$, sends $u+v \in \Gamma_0$ to $u+v \in A(u+v)$, and which acts on $\Gamma_1$ via the bijection $\Gamma_1 \leftrightarrow Q_1$. The kernel of this map is precisely the admissible ideal $I'$.
\end{proof}

\begin{proposition}\label{c.splitpres}  
Let $B = kQ/I$, and $A$ a maximal subalgebra of split type. Suppose that $A$ is $\Inn^*(B)$-conjugate to $A(u,v,U)$, for some $(u,v) \in V^2(Q)$ and codimension-$1$ subspace $U\le ukQ_1v$. Write $A \cong k\Gamma/I'$ for a certain quiver $\Gamma$ and admissible ideal $I'$. Then $\Gamma_0 = Q_0$, and for all $w,x \in \Gamma_0$, 

\begin{equation*}
\dim_kw\left(J(A)/J(A)^2\right)x = \dim_kw\left(J(A)/J(B)^2\right)x + \dim_kw\left(J(B)^2/J(A)^2\right)x.
\end{equation*} 
\noindent In particular: 

\begin{enumerate}
 \item For all $w \neq u$ and $x \neq v$, there are $w\left( J(B)/J(B)^2\right)x$ arrows from $w$ to $x$ in $\Gamma$,
 \item There are $\dim_ku\left(J(B)/J(B)^2\right)v -1$ arrows from $u$ to $v$, and  
 \item There are at least $\dim_ku\left(J(B)/J(B)^2\right)x$ (resp. $\dim_kw\left(J(B)/J(B)^2\right)v$) arrows from $u$ to $x$ (resp. from $w$ to $v$).  
 \end{enumerate}
 \noindent Furthermore, $J(B)^4 \subset J(A)^2$, so that any arrows in $\Gamma$ that do not appear as arrows in $Q$ arise from elements of $J(B)^2$ or $J(B)^3$.
\end{proposition}  

\begin{proof} 
 $A(u,v,U)/J(A(u,v,U)) = kQ_0$ implies that $\Gamma_0 = Q_0$. The dimension formula follows from the $kQ_0$-bimodule isomorphism $J(A)/J(A)^2/\left( J(B)^2/J(A)^2 \right) \cong J(A)/J(B)^2$, and claims (1)-(3) follow from the formula  

\begin{equation*}
J(A) = J(B)\cap A = U \oplus \left( \bigoplus_{(w,x)\neq (u,v)}{wkQ_1x} \right) \oplus J(B)^2.
\end{equation*}

\noindent The final claim follows from Corollary \ref{c.inherit} (3).
\end{proof}   

\noindent Although this corollary does not give us an explicit form for $I'$, we can use it to present maximal subalgebras of split type in the hereditary case, i.e. when $I = \{ 0\}$.  

\begin{proposition}\label{c.pathalg}
Let $B = kQ$ for an acyclic quiver $Q$, and $A \subset B$ a maximal subalgebra conjugate to $A(u,v,U)$, for some $(u,v) \in V^2(Q)$. Write $A \cong k\Gamma / I'$, for a finite quiver $\Gamma$ and admissible ideal $I' \subset k\Gamma$. Then $\Gamma_0 = Q_0$, and $\Gamma_1$ can be obtained from $Q$ as follows: 
\begin{enumerate} 
\item Replace the $|uQ_1v|$ arrows from $u$ to $v$ in $Q$ with $|uQ_1v|-1$ arrows, indexed by a fixed basis $\{ \alpha_1,\ldots , \alpha_d\}$ of $U$;
\item For each arrow $\gamma$ with target $u$, add an arrow $\overline{\gamma} : s(\gamma) \rightarrow v$; 
\item For each arrow $\gamma$ with source $v$, add an arrow $\underline{\gamma}: u \rightarrow t(\gamma)$.
\end{enumerate} 
\noindent Furthermore, $I'$ can be taken to be the ideal generated by the relations $\overline{\beta}\gamma - \beta\underline{\gamma}$, for all arrows $\beta$ and $\gamma$ in $Q$ with $t(\beta) = u$ and $s(\gamma) = v$. 
\end{proposition} 

\begin{proof} 
We may assume without loss of generality that $A = A(u,v,U)$. Find $\alpha_{d+1} \in ukQ_1v$ such that $U \oplus k\alpha_{d+1} = ukQ_1v$. Then for each $\gamma \in Q_1$ with $t(\gamma ) = u$, $\gamma\alpha_{d+1} \in J(A)\setminus J(A)^2$. If $w = s(\gamma)$, then clearly the paths of the form $\gamma \alpha_{d+1}$, along with the arrows in $Q_1$ from $w$ to $v$, form a basis for $w(J(A)/J(A)^2)v$. Define $\overline{\gamma} = \gamma\alpha_{d+1}$. A similar argument exhibits a basis for $u(J(A)/J(A)^2)x$, and allows us to define $\underline{\gamma} = \alpha_{d+1}\gamma$ if $\gamma$ is an arrow in $Q$ from $v$ to $x$. The form for $\Gamma$ then follows from Proposition \ref{c.splitpres}. For all arrows $\beta, \gamma \in Q_1$ with $t(\beta) = u$ and $s(\gamma) = v$, $\overline{\beta}\gamma = \beta\alpha_{d+1}\gamma = \beta\underline{\gamma}$. Hence, $\overline{\beta}\gamma-\beta\underline{\gamma}$ is in the kernel of the projection map $k\Gamma \rightarrow A$. If $I'$ is the ideal generated by these commutation relations, then it is straightforward to check that the induced $k$-algebra projection $k\Gamma/I' \rightarrow A$ has an inverse, so that the desired isomorphism holds.
\end{proof}  

\begin{example}\label{e.A4} 

Let $B = kQ$, where  
\[
Q=
\vcenter{\hbox{
\begin{tikzpicture}[point/.style={shape=circle,fill=black,scale=.3pt,outer sep=3pt},>=latex, scale = 2]
\node[point,label={below:$v_1$}] (1) at (0,0) {} ; 
\node[point,label={below:$v_2$}] (2) at (1,0) {}; 
\node[point,label={below:$v_3$}] (3) at (2,0) {}; 
\node[point,label={below:$v_4$}] (4) at (3,0) {}; 
\path[->] (1) edge node[above] {$\alpha$} (2) ; 
\path[->] (2) edge node[above] {$\beta$} (3) ; 
\path[->] (3) edge node[above] {$\gamma$} (4) ;
\end{tikzpicture} }}
\]  
\noindent is an equioriented Dykin quiver of type $\mathbb{A}_4$. Then any maximal subalgebra of separable type must be $\Inn^*(B)$-conjugate to one of the six bound quiver algebras displayed below: \newline

\begin{center}
$
\begin{array}{| c | c | c |} 

\hline k\Gamma / I & \Gamma & I \\ 

\hline A(v_1+v_2) & \begin{tikzpicture}[point/.style={shape=circle,fill=black,scale=.3pt,outer sep=3pt},>=latex, baseline=-3, scale =2] 
\node[point] (1) at (0,0) {} edge[in = -290, out = -250, loop] node[above] {$\alpha$} () ; 
\node[point] (2) at (1,0) {}; 
\node[point] (3) at (2,0) {}; 
\path[->] (1) edge node[above] {$\beta$} (2) ; 
\path[->] (2) edge node[above] {$\gamma$} (3) ; 
\end{tikzpicture} & (\alpha^2) \\

\hline A(v_1+v_3) & \begin{tikzpicture}[point/.style={shape=circle,fill=black,scale=.3pt,outer sep=3pt},>=latex, baseline=-3, scale=2]  
\node[point] (1) at (0,0) {} ; 
\node[point] (2) at (1,0) {} ; 
\node[point] (3) at (2,0) {} ; 
\path[->] (2) edge node[above] {$\gamma$} (1) ; 
\path[->] (2.25) edge[bend left=25] node[above] {$\alpha$} (3.25) ; 
\path[->] (3.-25) edge[bend right=-25] node[below] {$\beta$} (2.-25) ;
 \end{tikzpicture} & (\beta \alpha) \\ 

\hline A(v_1+v_4) & \begin{tikzpicture}[point/.style={shape=circle,fill=black,scale=.3pt,outer sep=3pt},>=latex, baseline=-3, scale=2]  
\node[point] (1) at (0,0) {} ; 
\node[point] (2) at (1,.5) {} ; 
\node[point] (3) at (2,0) {} ; 
\path[->] (1) edge node[above] {$\alpha$} (2) 
(2) edge node[above] {$\beta$} (3) 
(3) edge node[below] {$\gamma$} (1);
\end{tikzpicture}
 & (\gamma \alpha) \\ 

\hline A(v_2+v_3) & \begin{tikzpicture}[point/.style={shape=circle,fill=black,scale=.3pt,outer sep=3pt},>=latex, baseline=-3, scale=2]  
\node[point] (2) at (1,0) {} edge[in = -290, out = -250, loop] node[above] {$\beta$} () ; 
\node[point] (1) at (0,0) {}; 
\node[point] (3) at (2,0) {}; 
\path[->] (1) edge node[above] {$\alpha$} (2) ; 
\path[->] (2) edge node[above] {$\gamma$} (3) ; 
\end{tikzpicture}
 & (\beta^2) \\ 

\hline A(v_2+v_4) & \begin{tikzpicture}[point/.style={shape=circle,fill=black,scale=.3pt,outer sep=3pt},>=latex, baseline=-3, scale=2]   
\node[point] (1) at (0,0) {} ; 
\node[point] (2) at (1,0) {} ; 
\node[point] (3) at (2,0) {} ; 
\path[->] (1) edge node[above] {$\alpha$} (2) ; 
\path[->] (2.25) edge[bend left=25] node[above] {$\beta$} (3.25) ; 
\path[->] (3.-25) edge[bend right=-25] node[below] {$\gamma$} (2.-25) ;
\end{tikzpicture}
 & (\gamma \beta) \\ 

\hline A(v_3+v_4) & \begin{tikzpicture}[point/.style={shape=circle,fill=black,scale=.3pt,outer sep=3pt},>=latex, baseline=-3,scale=2] 
\node[point] (3) at (2,0) {} edge[in = -290, out = -250, loop] node[above] {$\gamma$} () ; 
\node[point] (2) at (1,0) {}; 
\node[point] (1) at (0,0) {}; 
\path[->] (1) edge node[above] {$\alpha$} (2) ; 
\path[->] (2) edge node[above] {$\beta$} (3) ; 
\end{tikzpicture} & (\gamma^2) \\ 
\hline

\end{array} 
$  
\end{center} 

\noindent Any maximal subalgebra of split type must be $\Inn^*(B)$-conjugate to one of the three bound quiver algebras displayed below:

\begin{center}  
$  
\begin{array}{| c | c | c |} 

\hline k\Gamma / I & \Gamma & I \\ 

\hline A(v_1,v_2,\{ 0 \} ) & \begin{tikzpicture}[point/.style={shape=circle,fill=black,scale=.3pt,outer sep=3pt},>=latex, baseline=-3,scale=2] 
\node[point,label={below:$v_1$}] (1) at (0,-.5) {} ;  
\node[point,label={above:$v_2$}] (2) at (0,.5) {} ; 
\node[point,label={above:$v_3$}] (3) at (1,0) {} ; 
\node[point,label={above:$v_4$}] (4) at (2,0) {} ; 
\path[->] (1) edge node[below] {$\underline{\beta}$} (3) 
(2) edge node[above] {$\beta$} (3) 
(3) edge node[above] {$\gamma$} (4);
\end{tikzpicture}
 & \{ 0 \} \\ 

\hline A(v_2,v_3,\{ 0 \} ) &  \begin{tikzpicture}[point/.style={shape=circle,fill=black,scale=.3pt,outer sep=3pt},>=latex, baseline=-3,scale=2]  
\node[point,label={left:$v_1$}] (1) at (0,0) {} ; 
\node[point,label={above:$v_3$}] (3) at (1,.5) {} ; 
\node[point,label={below:$v_2$}] (2) at (1,-.5) {} ; 
\node[point,label={right:$v_4$}] (4) at (2,0) {} ; 
\path[->] (1) edge node[above] {$\overline{\alpha}$} (3)
(1) edge node[below] {$\alpha$} (2) 
(3) edge node[above] {$\gamma$} (4) 
(2) edge node[below] {$\underline{\gamma}$} (4);
\end{tikzpicture}
 & (\overline{\alpha}\gamma - \alpha\underline{\gamma}) \\ 

\hline A(v_3,v_4.\{ 0 \} ) & \begin{tikzpicture}[point/.style={shape=circle,fill=black,scale=.3pt,outer sep=3pt},>=latex, baseline=-3,scale=2]  
\node[point,label={left:$v_1$}] (1) at (0,0) {} ; 
\node[point,label={above:$v_2$}] (2) at (1,0) {} ; 
\node[point,label={right:$v_3$}] (3) at (2,.5) {} ; 
\node[point,label={right:$v_4$}] (4) at (2,-.5) {} ; 
\path[->] (1) edge node[above] {$\alpha$} (2) 
(2) edge node[above] {$\beta$} (3) 
(2) edge node[below] {$\overline{\beta}$} (4);
\end{tikzpicture}
 & \{ 0 \} \\ 

\hline

\end{array}
$ 
\end{center}

\end{example}

\noindent {\bf{Note:}} If $I \neq \{ 0 \}$, then $I \subset J(Q)^2 \subset A$ allows us to consider relations in $I$ as ``generalized relations'' in $\Gamma$. We say ``generalized'' because these elements may not actually lie in $J(A)^2$. In other words, we can always realize $A$ as a \emph{generally non-admissible} quotient of $k\Gamma/I$. This level of detail is sufficient for the purposes of this paper. \newline 

\begin{example} 
Consider $B = kQ/I$, where  

\[
Q=
\vcenter{\hbox{
\begin{tikzpicture}[point/.style={shape=circle,fill=black,scale=.3pt,outer sep=3pt},>=latex, scale = 2]
\node[point,label={below:$v_1$}] (1) at (0,0) {} ; 
\node[point,label={below:$v_2$}] (2) at (1,0) {}; 
\node[point,label={below:$v_3$}] (3) at (2,0) {}; 
\node[point,label={below:$v_4$}] (4) at (3,0) {}; 
\path[->] (1) edge node[above] {$\alpha$} (2) ; 
\path[->] (2.75) edge node[above] {$\beta_1$} (3.75)  
(2.-75) edge node[below] {$\beta_2$} (3.-75); 
\path[->] (3) edge node[above] {$\gamma$} (4) ;
\end{tikzpicture} }}
\] 

\noindent and $I = (\alpha\beta_1-\alpha\beta_2 )$. Then the maximal subalgebra of $kQ$ corresponding to the triple $(v_2,v_3, k\beta_1)$ can be presented as $A= k\Gamma / I'$, where  

\[ 
\Gamma = 
\vcenter{\hbox{ 
 \begin{tikzpicture}[point/.style={shape=circle,fill=black,scale=.3pt,outer sep=3pt},>=latex, baseline=-3,scale=2]  
\node[point,label={left:$v_1$}] (1) at (0,0) {} ; 
\node[point,label={above:$v_3$}] (3) at (1,.5) {} ; 
\node[point,label={below:$v_2$}] (2) at (1,-.5) {} ; 
\node[point,label={right:$v_4$}] (4) at (2,0) {} ; 
\path[->] (1) edge node[above] {$\overline{\alpha}$} (3)
(1) edge node[below] {$\alpha$} (2) 
(3) edge node[above] {$\gamma$} (4) 
(2) edge node[below] {$\underline{\gamma}$} (4) 
(2) edge node[left] {$\beta_2$} (3); 
\end{tikzpicture}
}}
\] 

\noindent and $I' = (\overline{\alpha}\gamma - \alpha\underline{\gamma})$. Therefore, $A(v_2,v_3,k\beta_1) = A/I = A/(\overline{\alpha} - \alpha\beta_2)$. It follows that $\overline{\alpha}$ is in the radical square of $A(v_2,v_3,k\beta_1)$. Hence, to present $A(v_2,v_3, k\beta_1)$, we must actually bound the quiver

\[ 
\vcenter{\hbox{ 
 \begin{tikzpicture}[point/.style={shape=circle,fill=black,scale=.3pt,outer sep=3pt},>=latex, baseline=-3,scale=2]  
\node[point,label={left:$v_1$}] (1) at (0,0) {} ; 
\node[point,label={above:$v_3$}] (3) at (1,.5) {} ; 
\node[point,label={below:$v_2$}] (2) at (1,-.5) {} ; 
\node[point,label={right:$v_4$}] (4) at (2,0) {} ; 
\path[->] (1) edge node[below] {$\alpha$} (2) 
(3) edge node[above] {$\gamma$} (4) 
(2) edge node[below] {$\underline{\gamma}$} (4) 
(2) edge node[left] {$\beta_2$} (3); 
\end{tikzpicture}
}}
\] 

\noindent by the relation $\alpha\beta_2\gamma - \alpha \underline{\gamma}$. 

\end{example}  


\section{Type $\mathbb{A}$ Path Algebras}\label{s.main}

In this section we prove theorem \ref{t.main}. Unless otherwise stated, let $Q$ be a type-$\mathbb{A}$ Dynkin quiver on $n$ vertices, in other words a quiver whose underlying graph is of the form: 

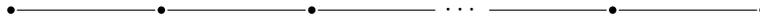
\begin{figure}[th]\label{f.dynkin}
\[ 
\vcenter{\hbox{ 
 \begin{tikzpicture}[point/.style={shape=circle,fill=black,scale=.3pt,outer sep=3pt},>=latex, baseline=-3,scale=2]  
\node[point] (1) at (0,0) {} ; 
\node[point] (2) at (1,0) {} ; 
\node[point] (3) at (2,0) {} ;  
\node (4) at (3,0) {$\cdots$} ;
\node[point] (5) at (4,0) {} ; 
\node[point] (6) at (5,0) {} ;
\path (1) edge  (2)  
(2) edge (3) 
(3) edge (4) 
(4) edge (5) 
(5) edge (6) ;
\end{tikzpicture}
}}
\]
\caption{A type $\mathbb{A}$ Dynkin diagram.}
\end{figure}

Let $B = kQ$. If $n = 2m$ is even, label the vertices of Figure 1 from left to right as $v_{-m}$, $v_{-m+1}$,$\ldots$, $v_{-1}$, $v_1$,$\ldots$, $v_{m-1}$, $v_m$. If $n = 2m+1$ is odd, label the vertices $v_{-m}$,$\ldots$, $v_{-1}$, $v_0$, $v_1$,$\ldots$, $v_m$ in an analogous manner. Whenever it is understood that we are talking about vertices, we may abbreviate ``$v_i$'' as ``$i$.'' We call $v_{-i}$ the \emph{predecessor} of $v_{-i+1}$, and $v_{-i+1}$ the \emph{successor} of $v_{-i}$. Note that if $n$ is even, $v_{-1}$ is the predecessor of $v_1$. If we need to talk about the predecessor (resp. successor) of a vertex $v$, we simply denote it by $\operatorname{pred}(v)$ (resp. $\operatorname{succ}(v)$). We recursively define $\operatorname{pred}^i(v)$ and $\operatorname{succ}^i(v)$ as follow: $\operatorname{pred}^1(v) = \operatorname{pred}(v)$ and $\operatorname{succ}^1(v) = \operatorname{succ}^1(v)$. If $\operatorname{pred}^i(v)$ and $\operatorname{succ}(v)$ have already been defined, then $\operatorname{pred}^{i+1}(v)$ is the predecessor of $\operatorname{pred}^i(v)$, and $\operatorname{succ}^{i+1}(v)$ is the successor to $\operatorname{succ}^i(v)$, whenever these vertices are defined. Let $\underline{w}$ denote the binary word of length $n-1$ such that $\underline{w}(i) = +1$ if the edge from $v_i$ to its successor starts at $v_i$, and $\underline{w}(i) = -1$ if the edge starts at the successor to $v_i$. We treat $\underline{w}$ as a function $\{ -m,\ldots , m-1\} \rightarrow \{ -1, +1\}$, or as an ordered $(m-1)$-tuple $\underline{w} = w_{-m}w_{-m+1}\cdots w_{m-1}$, where each $w_i \in \{ \pm 1\}$. Let $\alpha_i$ denote the edge between $v_i$ and its successor. To ease notation slightly, we use the following shorthand for maximal subalgebras of $kQ$:  \newline 

\begin{equation*} 
A_i := A(v_i, \operatorname{succ}(v_i), \{ 0\}) \text{ for all $i< m$,}
\end{equation*}  

\begin{equation*} 
A_{i,j} := A(v_i+v_j), \text{ for all $v_i, v_j \in Q_0$ with $i \neq j$.}
\end{equation*} \newline


Since any automorphism of $Q$ induces an automorphism of its underlying Dynkin diagram, $\Aut (Q) \le C_2$, the cyclic group of order 2. Define $\underline{w}^*$ to be the binary word of the quiver obtained from $Q$ by applying the unique non-identity automorphism of the underlying Dynkin graph to $Q$. In other words, $\underline{w}^*(i) = -\underline{w}(\operatorname{pred}(-i))$ for all $-m \le i \le m-1$. Clearly, $\underline{w}^{**} = \underline{w}$ and $\Aut (Q) = C_2$ if and only if $\underline{w}^* = \underline{w}$. The following lemma will be used extensively throughout our proofs:

\begin{lemma}\label{l.autos}
Let $i$ and $j$ be integers with $-m \le i,j \le m-1$. Suppose that $A_i$ has a connected Ext quiver. If $A_i \cong A_j$, then there exists an isomorphism $\psi : A_i \rightarrow A_j$ such that $\psi(Q_0) = Q_0$. If $\psi (\ell ) = -\ell$ for all $\ell \in Q_0$, then $\Aut (Q) = C_2$.
\end{lemma}  

\begin{proof}  
Let $\psi' : A_i \rightarrow A_j$ be any $k$-algebra isomorphism. Since $Q_0$ is a complete set of primitive orthogonal idempotents for both $A_i$ and $A_j$, $\psi' (Q_0)$ is a complete set of primitive orthogonal idempotents for $A_j$. Since $\Inn(A_j)$ acts transitively on such sets, there exists a $\iota_u \in \Inn(A_j)$ such that $\iota_u\psi' (Q_0) = \psi (Q_0)$. Setting $\psi = \iota_u\circ \psi'$ demonstrates the first claim. For the second, suppose that $\psi (\ell) = -\ell$ for all $\ell \in Q_0$. Then for all $\ell \neq i$, Proposition \ref{c.pathalg} implies $\dim_k\ell J(B)/J(B)^2\operatorname{succ}(\ell) = \dim_k \ell J(A_i)/J(A_i)^2 \operatorname{succ}(\ell) = \dim_k (-\ell) J(A_j)/J(A_j)^2 \operatorname{pred}(-\ell)$. But $\underline{w}(\ell) = +1$ if and only if $\dim_k\ell J(B)/J(B)^2 \operatorname{succ}(\ell) = 1$. But $Q$ is a tree: since $-\ell$ is adjacent to $\operatorname{pred}(-\ell)$ in $Q$ and $J(A_j) \subset J(B)$, the equality $\dim_k(-\ell)J(A_j)/J(A_j)^2\operatorname{pred}(-\ell )=1$ implies $\operatorname{pred}(-\ell)J(B)/J(B)^2(-\ell) = 0$ and $\underline{w}(\operatorname{pred}(-\ell)) = -1$. In other words, for all $\ell \neq i$ we have $\underline{w}(\ell) = - \underline{w}(\operatorname{pred}(-\ell) )$. Suppose by way of contradiction that $\underline{w}(i) = \underline{w}(\operatorname{pred}(-i))$. Without loss of generality we may assume $\underline{w}(i) = +1$. Then Proposition \ref{c.pathalg} and the assumption that $A_i$ has a connected Ext quiver together imply that at least one of the two conditions must hold: 
\begin{enumerate} 
\item $\operatorname{pred}(i)$ exists and $\underline{w}(\operatorname{pred}(i)) = +1$, or
\item $\operatorname{succ}^2(i)$ exists and $\underline{w}(\operatorname{succ}^2(i)) = +1$.
\end{enumerate} 
Suppose that the first condition holds. Then by Corollary \ref{c.inherit} (3), $\alpha_{\operatorname{pred}(i)}\alpha_i \in J(A_i)$ and the vector space $\operatorname{pred}(i)J(A_i)\operatorname{succ}(i)$ is non-zero. Therefore, the vector space $\operatorname{succ}(-i)J(A_j)\operatorname{pred}(-i)$ is non-zero as well. But note that since $Q$ is a type $\mathbb{A}$ Dynkin quiver, $\operatorname{succ}(-i)J(A_j)\operatorname{pred}(-i) = \operatorname{succ}(i)J(B)(-i)J(B)\operatorname{pred}(-i)$. But then in particular $(-i)J(B)\operatorname{pred}(-i) \neq 0$, which holds if and only if we have $\operatorname{pred}(-i)J(B)(-i) = 0$. In turn, this holds if and only if $\underline{w}(\operatorname{pred}(-i)) = -1$, a contradiction. So we must have $\underline{w}(i) = -\underline{w}(\operatorname{pred}(-i))$. If the second condition holds, a similar argument allows us to conclude again that $\underline{w}(i) = -\underline{w}(\operatorname{pred}(-i))$. But then we must have $\underline{w}^* = \underline{w}$ and hence $\Aut (Q) = C_2$.
\end{proof} 

To prove the main theorem, we prove it separately for maximal subalgebras of split type and separable type. For split type, we use Proposition \ref{c.pathalg} to distinguish three essential cases, depending on the form of the Ext quiver of $A_i$. If $w_{\operatorname{pred}(i)}w_iw_{\operatorname{succ}(i)} = (+1)(+1)(+1)$ or $(-1)(-1)(-1)$, then the Ext quiver will contain a commutative square: 

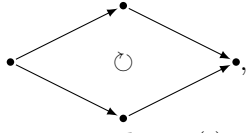
\begin{center}
 \begin{tikzpicture}[point/.style={shape=circle,fill=black,scale=.3pt,outer sep=3pt},>=latex, baseline=-3,scale=1.5]   
 \node[point] (i-1) at (0,0) {};  
 \node[point] (i) at (1,0.5) {};  
 \node (0) at (1,0) {$\circlearrowright$} ;
 \node[point,] (i+1) at (1,-0.5) {}; 
 \node[point] (i+2) at (2,0) {};  
 \path[->] (i-1) edge (i)  
 (i-1) edge (i+1) 
 (i) edge (i+2) 
 (i+1) edge (i+2);
\end{tikzpicture},
\end{center} 

\noindent where the two middle vertices are $i$ and $\operatorname{succ}(i)$. In all other cases, $A_i$ will be hereditary. Again, Proposition 3.3 implies that the underlying graph of the Ext quiver of $A_i$ has the form: 

\begin{figure}[th]
\[
\vcenter{\hbox{ 
 \begin{tikzpicture}[point/.style={shape=circle,fill=black,scale=.3pt,outer sep=3pt},>=latex, baseline=-3,scale=1.5]   
 \node[point,label={below:$-m$}] (-m) at (0,0) {}; 
 \node[point,label={below:$\operatorname{pred}(i)$}] (i-1) at (2,0) {};  
 \node (b1) at (1,0) {$\cdots$};
 \node[point,label={below:$i$}] (i) at (3,0) {}; 
 \node[point,label={above:$\operatorname{succ}(i)$}] (i+1) at (4,0) {}; 
 \node[point,label={above:$\operatorname{succ}^2(i)$}] (i+2) at (5,0) {};  
 \node (b2) at (6,0) {$\cdots$};
 \node[point,label={below:$m$}] (m) at (7,0) {}; 
 \path (-m) edge (0.5,0) 
 (1.5,0) edge (i-1) 
 (i-1) edge (i) 
 (i+1) edge (i+2) 
 (i+2) edge (5.5,0) 
 (6.5,0) edge (m); 
\path[dotted] (i-1) edge[bend left=25] (i+1) 
(i) edge[bend right=25] (i+2);
\end{tikzpicture},
}}
\] 
\caption{}\label{f.quiverform}
\end{figure}

\noindent where for $i = -m$ we take this graph to only include the edges to the right of $i$, for $i = m-1$ we delete the section containing $\operatorname{succ}^2(i)$, and where dotted lines indicate that an edge may or may not be present. By the connectivity hypothesis, both dotted edges cannot be absent. Note that $A_{-m+1}$ and $A_{m-2}$ may themselves be path algebras over type $\mathbb{A}$ Dynkin quivers, say if $w_{-m}w_{-m+1}w_{-m+2} = (+1)(+1)(-1)$. This will be our second case. Otherwise, this graph will necessarily contain a trivalent vertex, and a unique leaf adjacent to this trivalent vertex. This will be our third and final case.  

We now prove our first case, where $A_i$ is a non-hereditary algebra:



\begin{lemma}\label{l.splitnonher} 
Let $Q$ and $B$ be as before. Let $A$ be a non-hereditary maximal subalgebra of split type whose Ext quiver is connected. Then $\Iso (A,B) = \Aut_k(B)\cdot A$.
\end{lemma} 

\begin{proof} 
Suppose first $-m \le i \le -1$ and that $A_i$ is not hereditary. Then $i>-m$ and $w_{\operatorname{pred}(i)}w_iw_{\operatorname{succ}(i)} = (+1)(+1)(+1)$ or $(-1)(-1)(-1)$. By inspecting the full sub-quiver from $-m$ to $\operatorname{pred}(i)$ and the full sub-quiver from $\operatorname{succ}^2(i)$ to $m$, we conclude that $A_i \cong A_j$ implies either $j = i$ or $j = -i-1$. If $A_i \cong A_{-i-1}$, use Lemma \ref{l.autos} to find an isomorphism $\psi : A_i \rightarrow A_{-i-1}$ such that $\psi(Q_0) = Q_0$ and $\dim_k(uJ(A_i)/J(A_i)^2v) = \dim_k(\psi(u)J(A_{-i-1})/J(A_{-i-1})^2\psi(v))$ for all $u,v \in Q_0$. In other words, we may assume without loss of generality that $\psi$ induces an automorphism of the underlying quivers of $A_i$ and $A_{-i-1}$. But then we must have $\psi (m) = -m$, $\psi (\operatorname{pred}(i)) = \operatorname{succ}(-i)$, and $\psi(\operatorname{succ}^2(i)) = \operatorname{pred}(-i-1)$. This forces $\psi (\ell) = -\ell$ for all $\ell \in Q_0$, so that again by Lemma \ref{l.autos} we have $\Aut (Q) = C_2$. In this case, the non-identity automorphism of $Q$ sends $A_i$ to $A_{-i-1}$. Therefore, $\Aut_k(B)\cdot A_i = \Aut_k(B)\cdot A_{-i-1}$ and so $\Iso (A_i,B) = \Aut_k(B)\cdot A_i$. Otherwise $A_i\not\cong A_{-i-1}$, and again $\Iso (A_i,B) = \Aut_k(B)\cdot A_i$. The $i>-1$ case follows from replacing $\underline{w}$ with $\underline{w}^*$ and repeating the argument above. 
\end{proof} 

 We now prove our second case, where $A_i$ is hereditary but does not contain a trivalent vertex. We note again that this forces $i \in \{-m+1, m-2\}$.

\begin{lemma}\label{l.splitnotri} 
Let $Q$ and $B$ be as before, and let $A = A_{-m+1}$ or $A_{m-2}$. Suppose that $A$ is hereditary, that the Ext quiver of $A$ is connected, and that it does not contain a trivalent vertex. Then $\Iso (A,B) = \Aut_k(B)\cdot A$.
\end{lemma} 

\begin{proof} 
It suffices to prove the claim for $A_{-m+1}$. Without loss of generality, we may assume that $w_{-m+1} = +1$. For $1 \le m\le 3$ the claim can be verified through straightforward, but tedious computations. So we will assume $m>3$ throughout. By hypothesis, we must have $w_{-m}w_{-m+1}w_{-m+2} = (+1)(+1)(-1)$. From Proposition \ref{c.pathalg}, the Ext quiver of $A_{-m+1}$ is a type $\mathbb{A}$ Dynkin quiver, whose valence-$1$ vertices are $-m+1$ and $m$. The binary word associated to this quiver, choosing $-m+1$ to be the left-most vertex, is then simply $\underline{u} = (-1)(+1)(-1)w_{-m+3}\cdots w_{m-1}$. The only other $i$ for which $A_i$ can have a type $\mathbb{A}$ Dynkin quiver as its Ext quiver is $i = m-2$. Therefore, we have an inclusion $\Iso (A_{-m+1},B) \subset \Inn (B)\cdot A_{-m+1} \cup \Inn (B)\cdot A_{m-2}$. We claim that this inclusion is an equality if and only if $\Aut (Q) = C_2$. If it is not an equality then $A_{-m+1}\not\cong A_{m-2}$, and so in particular $\Aut (Q) = 1$. But then $\Iso (A,B) = \Inn (B)\cdot A = \Aut_k(B)\cdot A$, and the claim of the lemma is true in this case. So, suppose instead that $A_{m-2} \cong A$. Then $-m$ and $m-1$ are the leaves in the quiver of $A_{m-2}$. Comparing the full subquivers of $A_{m-2}$ and $A$ on $\{-m, -m+1, -m+2, -m+3\}$, we find that no isomorphism $A \rightarrow A_{m-2}$ can carry $-m+1$ to $-m$. But then this implies that we can find an isomorphism $\psi : A \rightarrow A_{m-2}$ which permutes $Q_0$ and carries $-m+1$ to $m-1$. This forces $\psi(j) = -j$ for all $j \in Q_0$. By Lemma \ref{l.autos} we have $\Aut (Q) = C_2$. In this case the non-identity element of $\Aut (Q)$, carries $A_{-m+1}$ to $A_{m-2}$, and so $\Iso (A,B) = \Aut_k(B)\cdot A$ in this case as well.
\end{proof} 

Now it only remains to prove the hereditary trivalent case. We break the proof into the cases where $|Q_0|$ is even or odd. For the next lemma, let $i$ be chosen such that $-m \le i \le m-1$. Then according to Figure \ref{f.quiverform}, exactly one of the dotted arrows must represent an arrow in the quiver of $A_i$. We refer to $A_i$ as $L_i$ if the arrow from $\operatorname{pred}(i)$ to $\operatorname{succ}(i)$ is present, and $R_i$ if the arrow from $i$ to $\operatorname{succ}^2(i)$ is present. For brevity, we write $A_i = L_i$ if the former holds, and $A_i = R_i$ if the latter holds. In either case, there is a unique trivalent vertex in the quiver of $A_i$, and a unique univalent vertex adjacent to it. We refer to this univalent vertex as the \emph{root} of the quiver. We call the smallest connected full subquiver containing this univalent vertex and $-m$ as the \emph{left path}. Similarly, the smallest connected full subquiver containing the univalent vertex and $m$ is called the \emph{right path}. Note that the left path and right path are just type $\mathbb{A}$ Dynkin quivers. The \emph{length} of the left/right path is just the number of arrows in it. 

\begin{lemma}\label{l.splittri} 
Let $G$ and $B$ be as before. Let $A$ be a hereditary maximal subalgebra of split type, whose Ext quiver is connected and contains a trivalent vertex. Then $\Iso (A,B) = \Aut_k(B)\cdot A$.
\end{lemma}


\begin{proof}  
\emph{Case 1:} Suppose that $n$ is even. Then in particular, $\operatorname{succ}(-1) = 1$. We start by showing that for any $i$ and $j$ with $-m \le i , j \le -1$ and $i \neq j$, $A_i \not\cong A_j$. To see this, first note that for such $i$, the length of the left path in $L_i$ is $i+m$, and the length of the right path is $m-i-1$. Similarly, length of the left path in $R_i$ is $m+i+2$, and the length of the right path is $m+i-2$. In particular, the difference between the lengths of the left and right paths in $L_i$ is odd, whereas the difference between the left and right paths in $R_i$ is even. It follows that if either $A_i = L_i$ and $A_j = R_j$, or $A_i = R_i$ and $A_j = R_j$, then $A_i \not\cong A_j$. Furthermore, $L_i \not\cong L_j$, since for $i \neq j$, the corresponding sets of left and right path lengths are disjoint. It follows that if $A_i \cong A_j$ for $i,j \le -1$, then necessarily $A_i = R_i$ and $A_j = R_j$. In fact, by inspection of path lengths, the only possibility is $m \geq 3$ and $\{ i,j\} = \{ -1,-3\}$. Without loss of generality, we may assume that $w_{-1} = (+1)$. Then necessarily $w_{-2}w_{-1}w_1 = (-1)(+1)(+1)$, so that $1$ is the root of $A_{-1}$ and it is a source in the quiver of $A_{-1}$. Since $A_{-3} = R_{-3}$ and $w_{-2} = -1$, we must also have $w_{-3} = -1$. But then $-2$ is the root of $A_{-3}$, and it is a sink in its quiver. This is a contradiction, and so $A_{-1}\not\cong A_{-3}$. 

The above argument implies that if $-m \le i \le -1$ and $j$ is chosen such that $A_i \cong A_j$, then $j \geq 1$. Note that if $j \geq 1$, then the lengths of the left and right paths of $L_j$ have an even difference, whereas they have an odd difference in $R_j$. Comparing path lengths, we find the following: 
\begin{enumerate} 
\item If $i \in \{ -1,-3\}$ and $A_i = R_i$, then $A_i \cong A_j$ implies $A_j = L_j$ and $j \in \{ -1,2\}$.
\item If $i \not\in \{ -1,-3\}$ and $A_i = R_i$, then $A_i \cong A_j$ implies $A_j = L_{-i-1}$. 
\item If $A_i = L_i$, then $A_i \cong A_j$ implies $A_j = R_{\operatorname{pred}(-i)}$.  
\end{enumerate} 
We have already shown that $A_{-1}\not\cong A_{-3}$. Similar computations show that if the root of $R_{-1}$ (resp. $R_{-2}$) is a source, then the root of $L_2$ (resp. $L_1$) is a sink, and vice-versa. We conclude $R_{-1} \not\cong L_2$ and $R_{-2}\not\cong L_1$. The remaining cases from statements (1)-(3) can be rephrased as follows: for all $i$, either $\Iso (A_i,B) = \Inn (B)\cdot A_i$ or $\Iso (A_i,B) = \Inn (B)\cdot A_i \cup \Inn (B)\cdot A_{\operatorname{pred}(-i)}$. If $\Iso (A_i,B) = \Inn (B)\cdot A_i$, there is nothing to show. So suppose $i \neq \operatorname{pred}(-i)$ and $A_i \cong A_{\operatorname{pred}(-i)}$, for some negative integer $i$. In particular $i \neq -1$, and if $i = -2$ we may assume $A_{-2} = L_{-2}$. Note that this implies that the left path of $A_i$ must have a shorter length than the right path, whereas the left path of $A_{-i-1}$ must have a longer length that its right path. Since any isomorphism $A_i \rightarrow A_{-i-1}$ permuting $Q_0$ must send the trivalent vertex (resp. root) of $A_i$ to the trivalent vertex (resp. root) of $A_{-i-1}$, it follows that such an isomorphism satisfies $j \mapsto -j$ for all $j \in Q_0$. By Lemma \ref{l.autos} we have $\Aut (Q) = C_2$, a contradiction (since $n$ is even). Hence $\Iso (A_i,B) = \Inn(B)\cdot A_i$, as we wished to show.



\emph{Case 2:}  Suppose that $n$ is odd. Then $Q_0$ is just the interval $[-m,m]$ of $\ZZ$. Let $i$ be chosen such that $-m \le i \le -1$. The the left path of $L_i$ has length $i+m$, and the right path has length $m-i+1$. The left path of $R_i$ has length $m+i+2$, and the right path has length $m-i-1$. As before, we want to start by showing that if $j$ is any number between $-m$ and $-1$ with $i \neq j$, then $A_i \not\cong A_j$. Suppose that $A_i = L_i$. Then by comparing path sizes, we see that if $A_i \cong A_j$, then $A_j = R_{i-2}$. Without loss of generality, we may assume $w_i = +1$. Then $w_{i-1}w_iw_{i+1} = (+1)(+1)(-1)$. But since $A_{i-2} = R_{i-2}$, we must have $w_{i-2} = +1$ as well. But then the root of $A_i$ is $i$, which is a sink, whereas the root of $A_{i-2}$ is $i-1$, which is a source. This is a contradiction, and so $A_i \not\cong A_j$ if $A_i = L_i$. Otherwise $A_i = R_i$. If $A_i \cong A_j$ and $A_j = L_j$ we are in the previous case, so assume $A_j = R_j$ as well. By comparing path lengths, we find that the only possibility is $\{ i,j\} = \{ -1,-2\}$. Suppose without loss of generality that $w_{-1} = +1$. Then since $A_{-1} = R_{-1}$ and is connected hereditary, we must have $w_{-2}w_{-1}w_{0} = (-1)(+1)(+1)$. But then $A_{-2} \neq R_{-2}$, since $A_{-2} = R_{-2}$ requires $w_{-2}$ and $w_{-1}$ to have the same sign. Hence $R_{-1} \not\cong R_{-2}$, and it follows that for all $-m \le i,j \le -1$, $A_i \not\cong A_j$. 

Now we show that if $A_i \cong A_j$ for $j\geq 0$, then $j = {-i-1}$ and $\Aut (Q) = C_2$. Suppose first that $A_i = L_i$. Then the other algebras $A_j$ for $j\geq 0$ which have the same set of path lengths as $A_i$ are $R_{-i-1}$ and $L_{-i+1}$. We first show that $L_i \not\cong L_{-i+1}$. Without loss of generality, we let $w_i = +1$. The root of $L_i$ is $i$, and the right path has a larger length than the left path. Treating the right path as a type $\mathbb{A}$ Dynkin quiver starting at $i$, its associated binary word is $(-1)(+1)w_{i+1}\cdots w_{m-1}$. Now, the larger path in $L_{-i+1}$ is the left path, and its root is vertex $-i+1$. Since the root of $L_i$ is a sink, $-i+1$ must be a sink too. This forces $w_{-i}w_{-i+1}w_{-i+2} = (+1)(+1)(-1)$. Treating the left path of $L_{-i+1}$ as a type $\mathbb{A}$ Dynkin quiver starting at $-i+1$, its associated binary word is $(-1)(-w_{-i-1})\cdots (-w_{-m})$. If $L_i \cong L_{-i+1}$, then we must have $(-1)(+1)w_{i+1}\cdots w_{m-1} = (-1)(-w_{-i-1})\cdots (-w_{-m})$. In particular, $w_{-i-1} = -1$ and for all $i+1 \le j \le m-1$, we have $w_j = -w_{-(j+1)}$. But then setting $j = -i$ we find $+1= w_{-i} = -w_{-(-i+1)} = -w_{i-1} = -(+1) = -1$, a contradiction. So $L_i\not\cong L_{-i+1}$, as we wished to show. Hence, we assume $L_i \cong R_{-i-1}$. Since the root of $L_i$ is a sink, it must be a sink in $R_{-i-1}$ as well. This implies that if we start at the root, the longer path in $R_{-i-1}$ has binary word $(-1)(+1)(-1)(-w_{-i-2})\cdots w_{-m}$. Therefore, we have $(-1)(+1)(-w_{-i-2})\cdots (-w_{-m}) = (-1)(+1)(w_{i+1})\cdots w_{m-1}$. This implies that $w_j = -w_{-(j+1)}$ for $i+1 \le j \le m-1$. Similarly, the binary word for the short path in $R_{-i-1}$ is $w_{-i}w_{-i+1}\cdots w_{m-1}$, whereas it is $(-w_{i-1})(-w_{i-2})\cdots (-w_{-m})$ for $L_i$. Hence, $w_j = -w_{-(j+1)}$ for $-m \le j \le i-1$. But $w_i= +1$ by hypothesis, and $w_{-i-1} = -1$ since $R_{-i-1}$ is a subalgebra of $B$ isomorphic to $L_i$. In other words, $\underline{w}^* = \underline{w}$, and $\Aut (Q) = C_2$. 

This shows that if $A_i \cong A_j$ for $j \geq 0$, then $j = -i-1$ and $\Aut (Q) = C_2$ for the $A_i = L_i$ case. Note that by replacing $\underline{w}$ by $\underline{w}^*$ if necessary, it only remains to consider the case when $A_i = R_i$ and $A_j = R_j$. By comparison of path lengths, $A_j = R_{-i-3}$. Note that we may assume $i \le -3$, for otherwise this reduces to a case that has been previously discussed. Comparing the binary words for the longer paths in $R_i$ and $R_{-i-3}$, we see $w_{i+1}\cdots w_{m-1} = (+1)(-1)(-w_{-i-4})\cdots (-w_{-m})$. Therefore, $w_j = -w_{-(j+1)}$ for $i+3 \le j \le m-1$. Since $i \le -3$, $i+3 \le 0$ and so this suffices to show that $\underline{w}^* = \underline{w}$. But since $w_i = +1$ and $A_i = R_i$, $w_{i+1} = +1$. Since we also assume $A_{-i-3} = R_{-i-3} \cong R_i$, comparison of roots yields $w_{-i-2} = +1$. But then $w_{i+1} \neq -w_{-(i+1+1)}$, and so $\underline{w}^* \neq \underline{w}$. This contradiction shows that $R_i \not\cong R_{-i-3}$. 

Putting this all together, we see that if $A_i \cong A_j$, then $j = -i-1$ and $\Aut (Q) = C_2$. But in this case $A_i$ and $A_{-i-1}$ lie in the same $\Aut_k(B)$-orbit, and so the lemma is proved.
\end{proof} 

Lemmas \ref{l.splitnotri}, \ref{l.splitnonher}, and \ref{l.splittri} combine to prove the following proposition: 

\begin{proposition}\label{p.splitorbit} 
Let $Q$ and $B$ be as before. If $A \subset B$ is a maximal subalgebra of split type, and the Ext quiver of $A$ is connected, then $\Iso (A,B) = \Aut_k(B)\cdot A$.
\end{proposition}

To finish the proof of Theorem \ref{t.main}, we need to show that $\Iso (A_{i,j}, B) = \Aut_k(B)\cdot A_{i,j}$ for all $i$ and $j$. If $i<j$, then Proposition \ref{p.seppres} asserts that the underlying graph of the Ext quiver of $A_{i,j}$ looks as follows: \newline

\begin{figure}[th]
\[
\vcenter{\hbox{ 
 \begin{tikzpicture}[point/.style={shape=circle,fill=black,scale=.3pt,outer sep=3pt},>=latex, baseline=-3,scale=1.5]   
 \node[point,label={below:$-m$}] (-m) at (0,0) {}; 
 \node (b1) at (1,0) {$\cdots$};
 \node[point,label={below:$\operatorname{pred}(i)$}] (i-1) at (2,0) {}; 
 \node[point,label={below:$i+j$}] (i+j) at (3,0) {};   
 \node (b2) at (2.5,0.5) {}; 
 \node (b3) at (3,1) {}; 
 \node (b4) at (3.5,0.5) {}; 
 \node[point,label={below:$\operatorname{succ}(j)$}] (j+1) at (4,0) {}; 
 \node (b5) at (5,0) {$\cdots$} ; 
 \node[point,label={below:$m$}] (m) at (6,0) {}; 
 
 \path (-m) edge (0.5,0) 
 (1.5,0) edge (i-1) 
 (i-1) edge (i+j) 
 (i+j) edge[bend left=25] (b2) 
 (i+j) edge[bend right=25] (b4) 
 (i+j) edge (j+1) 
 (j+1) edge (4.5,0) 
 (5.5,0) edge (m);
 
\path[dotted] (b2) edge[bend left=25] (b3) 
(b3) edge[bend left=25] (b4);
\end{tikzpicture},
}}
\] 
\caption{}\label{f.sepform}
\end{figure}

\noindent where if $i = -m$ (resp. $j = m$) we delete the subgraph including $\operatorname{pred}(i)$ and all vertices to its left (resp. $\operatorname{succ}(j)$ and all vertices to its right). Examining the paths from $-m$ to $i+j$ and $i+j$ to $m$, we see $\Iso (A_{i,j},B) \subset \Aut_k(B)\cdot A_{i,j} \cup \Aut_k(B)\cdot A_{-i,-j}$. In particular, $\Iso (A_{i,j},B) = \Aut_k(B)\cdot A_{i,j}$ if $i=-j$, and so without loss of generality we may replace the $i<j$ assumption by the assumption that $|i|<|j|$ throughout (allowing now for the possibility that $i\geq j$). If $\Aut (Q) = C_2$, then the unique non-identity automorphism of $Q$ extends to an automorphism of $B$ which sends $v_i\mapsto v_{-i}$, $v_j \mapsto v_{-j}$. Therefore $\Iso (A_{i,j},B) = \Aut_k(B)\cdot A_{i,j}$ in this case as well, and so we may also assume $\Aut (Q) = 1$.

Under these assumptions, we will be done if we can show $A_{i,j} \not\cong A_{-i,-j}$ for the remaining cases. In other words, we must prove the following lemma: 

\begin{lemma}\label{l.evenodd} 
Under the assumptions above, $A_{i,j} \not\cong A_{-i,-j}$.
\end{lemma}

 \noindent Again we break this up into two arguments, depending on whether $n$ is even or odd. First suppose that $Q$ has $n=2m$ vertices. Note that if $i,j \le -1$ or $i,j \geq 1$, then $A_{i,j} \not\cong A_{-i,-j}$ by examination of the edge between $v_{-1}$ and $v_1$ in the quivers of $A_{i,j}$ and $A_{-i,-j}$. Therefore, $\Iso (A_{i,j},B) = \Inn(B)\cdot A_{i,j}$ in this case. Hence, we may assume without loss of generality that $-m \le i \le -1 < 1 \le j \le m$, so that combined with our reductions above, we have $1 \le -i < j \le m$. Write $\underline{w} = \underline{w}_1\cdot\underline{w}_2\cdot\underline{w}_3\cdot\underline{w}_4\cdot\underline{w}_5$, where the $\underline{w}_i$ are defined as follows (see Figure \ref{f.wdef}):  
 
 \begin{enumerate} 
 \item $\underline{w}_1$ is the binary word from $v_{-m}$ to $v_{-j}$, 
 \item $\underline{w}_2$ is the binary word from $v_{-j}$ to $v_i$, 
 \item $\underline{w}_3$ is the binary word from $v_i$ to $v_{-i}$, 
 \item $\underline{w}_4$ is the binary word from $v_{-i}$ to $v_j$, and 
 \item $\underline{w}_5$ is the binary word from $v_j$ to $v_m$.
 \end{enumerate}   
 
 \begin{figure}[th]
\[
\vcenter{\hbox{ 
 \begin{tikzpicture}[point/.style={shape=circle,fill=black,scale=.3pt,outer sep=3pt},>=latex, baseline=-3,scale=1.5]   
 \node[point,label={below:$-m$}] (-m) at (0,0) {}; 
 \node (b1) at (1,0) {$\cdots$}; 
 \node (w1) at (1,-0.25) {$\underline{w}_1$};
 \node[point,label={below:$-j$}] (i-1) at (2,0) {};  
 \node (w2) at (2.5,-0.25) {$\underline{w}_2$};
 \node[point,label={below:$i+j$}] (i+j) at (3,0) {};   
 \node (b2) at (2.5,0.5) {};  
 \node (w3) at (2.25,0.5) {$\underline{w}_3$};
 \node[point,label={above:$-i$}] (b3) at (3,1) {}; 
 \node (b4) at (3.5,0.5) {};  
 \node (w4) at (3.75,0.5) {$\underline{w}_4$};
 \node (b5) at (4,0) {$\cdots$} ;  
 \node (w5) at (4,-0.25) {$\underline{w}_5$};
 \node[point,label={below:$m$}] (m) at (5,0) {}; 
 
 \path (-m) edge (b1) 
 (b1) edge (i-1) 
 (i-1) edge (i+j) 
 (i+j) edge[bend left=25] (b2) 
 (i+j) edge[bend right=25] (b4) 
 (i+j) edge (b5) 
 (b5) edge (m);
 
\path[dotted] (b2) edge[bend left=25] (b3) 
(b3) edge[bend left=25] (b4);
\end{tikzpicture},
}}
\] 
\caption{}\label{f.wdef}
\end{figure}

\newpage
 


\begin{proofEven}  
Suppose there was an isomorphism $A_{i,j} \cong A_{-i,-j}$ by way of contradiction. Then the quivers of $A_{i,j}$ and $A_{-i,-j}$ must be isomorphic. Since $|i|<|j|$, this can only be true if the following three conditions hold: 
\begin{enumerate} 
\item $\underline{w}_1 = (\underline{w}_5)^*$,
\item $\underline{w}_2 = (\underline{w}_4)^*$, and 
\item either $\underline{w}_3\cdot\underline{w}_4 = \underline{w}_3^*\cdot\underline{w}_2^*$ or $\underline{w}_3\cdot \underline{w}_4 = (\underline{w}_3^*\cdot \underline{w}_2^*)^* = \underline{w}_2\cdot \underline{w}_3$.
\end{enumerate} 

\noindent If $\underline{w}_3\cdot \underline{w}_4 = \underline{w}_3^*\cdot \underline{w}_2^*$, then by applying condition (2) and cancelling $\underline{w}_4$ from both sides, $\underline{w}_3 = \underline{w}_3^*$. Then $\underline{w}^* = ( \underline{w}_1\cdot \underline{w}_2 \cdot \underline{w}_3 \cdot \underline{w}_4 \cdot \underline{w}_5)^* = \underline{w}_5^* \cdot \underline{w}_4^* \cdot \underline{w}_3^*\cdot \underline{w}_2^* \cdot \underline{w}_1^* =  \underline{w}_1\cdot \underline{w}_2 \cdot \underline{w}_3 \cdot \underline{w}_4 \cdot \underline{w}_5 = \underline{w}$, a contradiction. Hence, $\underline{w}_3\cdot \underline{w}_2^* = \underline{w}_2\cdot \underline{w}_3$.

\emph{Case 1:} Suppose that the length of $\underline{w}_2$ is less than or equal to the length of $\underline{w}_3$. Then since $\underline{w}_3\cdot \underline{w}_2^* = \underline{w}_2\cdot \underline{w}_3$, we may write $\underline{w}_2 = \e_1\cdots \e_d$, $\underline{w}_3= \e_1\cdots \e_d\e_{d+1}\cdots \e_{d+f}$, where each $\e_a \in \{ \pm 1\}$. Then  

\begin{equation}
(\e_1\cdots \e_d\e_{d+1}\cdots \e_{d+f})(-\e_d)\cdots (-\e_1)= (\e_1\cdots \e_d)(\e_1\cdots\e_d\e_{d+1}\cdots \e_{d+f}).  
\end{equation} 

Suppose that $d > f$. Then by comparing terms, we have $\e_{d+a} = \e_a$ for all $1 \le a \le f$ and $\e_{f+b} = -\e_{d-b+1}$ for all $1 \le b \le d$. Since $d+f$ is odd, $(d+f+1)/2$ is a natural number, and $f<d$ implies that $(d+f+1)/2 \le d$. Now, $d-b+1 = (d+f+1)/2$ precisely when $b = (d-f+1)/2$. Then $\e_{f+b} = -\e_{d-b+1}$ yields $\e_{(d+f+1)/2} = -\e_{(d+f+1)/2}$ for this value of $b$, a contradiction. Note that $d \neq f$ since $d+f$ is odd, and so the only remaining possibility is that $d< f$. After cancelling the $\e_1\cdots \e_d$ term from both sides, we find that 

\begin{equation} 
\e_{d+1}\cdots \e_{d+f}(-\e_d)\cdots (-\e_1) = (\e_1\cdots \e_f )(\e_{f+1} \cdots \e_{d+f}).
\end{equation}

\noindent For $1\le a \le f$, define $q = q(a)$ to be the largest non-negative integer such that $dq+a \le f$. Now, comparing terms on the left-hand-side and the right-hand-side of the above equation, we see $\e_a = \ldots  = \e_{dq+a} = \e_{d(q+1)+a}$. Since $d(q+1)+a > f$, it follows that $\e_{d(q+1)+a} = -\e_{d+f+1-(d(q+1)+a)} = -\e_{f-qd-a+1} = -\e_{f-(dq+a-1)}$.  Notice that $q = \lfloor \frac{f-a}{d} \rfloor$. If we can choose $a$ such that $a = f-(dq+a-1)$, we will have the desired contradiction. 

\emph{Sub-Case 1(a):} Say that $f$ is odd. Then $d$ is even, and for all $s$, $(f-ds+1)/2$ is an integer. Define $r$ to be the largest positive integer such that $dr+1 \le f$. Then $f < d(r+1)+1 < d(r+2)+1$. Let $a = (f-dr+1)/2$. It is clear that $a \le f$. The inequality $dr+1 \le f < d(r+2)+1$ can be re-arranged to say that $r \le \frac{f-\frac{f-rd+1}{2}}{d} < r+1$. In other words, $\lfloor \frac{f-a}{d} \rfloor = r$. But $a = f-dr-a+1 = f-d\lfloor \frac{f-a}{d} \rfloor -a + 1$ by definition, and we have found a choice of $a$ which works.  

\emph{Sub-Case 1(b):} Say that $f$ is even. Then $d$ is odd, and for all odd $s$, $(f-sd+1)/2$ is an integer. Let $r$ be the same as in Sub-Case 1(a). If $r$ is odd, then again $a = (f-rd+1)/2$ works. Otherwise $r$ is even, so that $r-1$ is odd and $a = (f-d(r-1)+1)/2$ is an integer. Notice that $d(r-1)+1 \le f < d((r-1)+2)+1$, so that $\lfloor \frac{f-a}{d} \rfloor = r$. Now, for this choice of $a$, we have $a = f-d(r-1) - a + 1 = (f - d\lfloor\frac{f-a}{d}\rfloor -a + 1) + d$. But since $r \geq 2$, $d(r-1)+a - 1 \geq 0$ and $f-d(r-1)-a+1 \le f$. Hence, $-\e_{f-(dr+a-1)} =  -\e_{f-(dr+a-1)+d} = -\e_{f-(d(r-1)+a-1)}$, and we obtain our desired contradiction. The proof of Case 1 is now complete.  

\emph{Case 2:} Suppose that the length of $\underline{w}_2$ is greater than or equal to the length of $\underline{w}_3$. Note that they cannot be equal, for otherwise $\underline{w}_3\cdot \underline{w}_2^* = \underline{w}_2\cdot \underline{w}_3$ would imply $\underline{w}_3 = \underline{w}_2$, and hence $\underline{w}_3^* = \underline{w}_3$, a contradiction. So then, we may write $\underline{w}_3 = \e_1\cdots \e_d$ and $\underline{w}_2 = \e_1\cdots \e_d\e_{d+1}\cdots \e_{d+f}$, for some $d,f>0$. As before, $d \neq f$. Hence, our equation reads 

\begin{equation} 
\e_1\cdots \e_d(-\e_{d+f})\cdots (-\e_{d+1})(-\e_d)\cdots (-\e_1) = \e_1\cdots \e_d\e_{d+1}\cdots \e_{d+f}\e_1\cdots \e_d.
\end{equation} 

\noindent Comparing the last $d$-terms on both sides of this equation, we conclude that $\e_a = -\e_{d-a+1}$, for $1 \le a \le d$. Now, $d$ is odd, so $(d+1)/2$ is an integer $\le d$. Therefore, for $a = (d+1)/2$ we have $\e_{(d+1)/2} = -\e_{d-(d+1)/2 +1} = -\e_{(2d - d - 1 + 2)/2} = -\e_{(d+1)/2}$, a contradiction.
\end{proofEven}   

 The remaining case left to consider is when $Q$ has $n = 2m+1$ vertices, $m \geq 1$. Again, we assume $\Aut (Q) = 1$ and $|i| < |j|$. If $i,j \le 0$ or $i,j \geq 0$ and $A_{i,j} \cong A_{-i,-j}$, then inspection of the long path in Figure \ref{f.sepform} implies that $\underline{w}(\ell) = -\underline{w}(\operatorname{pred}(-\ell ))$ for either $-m \le \ell \le 0$ or $0 \le \ell \le m-1$. In either case we conclude $\underline{w}^* = \underline{w}$, contradicting the triviality of $\Aut (Q)$. So we can assume $i<0 $ and $0\le j$ in addition to $|i| < |j|$. Then we can decompose $\underline{w}$ into $\underline{w} = \underline{w}_1\cdot \underline{w}_2 \cdot \underline{w}_3 \cdot \underline{w}_4 \cdot \underline{w}_5$ as in Figure \ref{f.wdef}. Under the assumptions that $A_{i,j} \cong A_{-i,-j}$ and $n$ is odd, the length of $\underline{w}_3$ must be even, and $\underline{w}_3\cdot \underline{w}_2^* = \underline{w}_2 \cdot \underline{w}_3$.

With this modified setup, we can now finish the proof of Lemma \ref{l.evenodd}:



\begin{proofOdd}
 \emph{Case 1:} Say that the length of $\underline{w}_3$ is greater than the length of $\underline{w}_2$. Write $\underline{w}_2 = \e_1\cdots \e_d$ and $\underline{w}_3 = \e_1\cdots \e_d\e_{d+1}\cdots \e_{d+f}$. Again, we find that equation (2) holds, and so $\e_{d+a} = \e_a$ for $1 \le a \le f$ and $\e_{f+b} = -\e_{d-b+1}$ for $1\le b \le d$. 

 \emph{Sub-Case 1(a):} Suppose that $f\le d$. If $d$ and $f$ are both odd, then $a= (f+1)/2$ is an integer $\le f$. Since $d+a\geq f+a > f$, we have $\e_a = \e_{d+a} = \e_{f+(d+a-f)} = -\e_{d-(d+a-f)+1} = -\e_{f-a+1}$. But for this choice of $a$, $a = f-a+1$, and we arrive at a contradiction. So we may assume that $d$ and $f$ are both even. Applying the same logic, we find that for all $a \le f$, we have a sequence of equalities: $\e_a = \e_{d+a} = -\e_{f-a+1} = -\e_{d+f-a+1} = \e_a$. Since $f$ and $d$ are both even, $\{ a,d+a\} \cap \{f-a+1, d+f-a+1\} = \emptyset$. The equalities $\e_a = -\e_{d+f - a + 1}$ and $\e_{d+a} = -\e_{f-a+1}$ are equivalent to the statement that $\underline{w}_3^* = \underline{w}_3$, which contradicts our hypotheses. Hence, the $f\le d$ case is complete. 

 \emph{Sub-Case 1(b):} Suppose that $f>d$. For any $a \le f$, define $q = q(a)$ to be the largest non-negative integer such that $dq + a \le f$. Then $\e_a = \ldots = \e_{dq+a} = \e_{d(q+1)+a} = -\e_{f-(dq+a-1)}$. If $f$ and $d$ are odd, then define $r$ to be the largest positive integer such that $dr +1 \le f$. If $r$ is even, then $a = (f-dr+1)/2$ is an integer $\le f$ which satisfies $\lfloor \frac{f-a}{d} \rfloor = r$ and $i = f-dr+a-1$. This implies $\e_a = -\e_a$ as before. Otherwise $r$ is odd. Then $a = (f-d(r-1)+1)/2$ satisfies $\lfloor\frac{f-a}{d}\rfloor = r$, $f-d(r-1)-a+1 \le f$, and $a=(f - dr -a + 1) + d$, which implies $\e_a = -\e_{f-d(r-1) -a +1} = -\e_a$ as before. It remains to consider the case that $f$ and $d$ are even. Then we have $\e_a = \ldots = \e_{dq+a} = \e_{d(q+1)+a} = -\e_{f-(dq+a-1)} = \ldots = -\e_{f-a+1} = -\e_{d+f-a+1}$ for all $1 \le a \le f$, so $\e_a = -\e_{d+f - (a-1)}$. Finally, for each $1 \le a \le d$, there is a unique $f< s \le f+d$ with $s \equiv a$ $(\operatorname{mod}$ $d)$. But then, $s = d(q+1) + a$, and the equality $\e_{d(q+1)+a} = -\e_{f-(dq+a-1)}$ tells us that $\e_b = -\e_{d+f-(b-1)}$ for all $1 \le b \le f+d$. This implies $\underline{w}_3 = \underline{w}_3^*$, contrary to our assumption that $\Aut (Q) = 1$. Hence, the $f>d$ case is complete.  

\emph{Case 2:} Say that the length of $\underline{w}_3$ is less than or equal to the length of $\underline{w}_2$. As before, $\ell (\underline{w}_3) \neq \ell (\underline{w}_2)$, so we may assume $\ell (\underline{w}_3) < \ell (\underline{w}_2)$. Write $\underline{w}_3 = \e_1\cdots \e_d$ and $\underline{w}_2 = \e_1\cdots \e_d\e_{d+1}\cdots \e_f$. As before, equation (3) holds. If $d$ is odd, repeat the same argument given for the even-vertices case. Otherwise $d$ is even, and so the equation $\e_a = -\e_{d-a+1}$ for $1 \le a \le d$ tells us $\underline{w}_3^* = \underline{w}_3$, contrary to our hypotheses. 
 \end{proofOdd}
 
 \noindent Lemma \ref{l.evenodd} immediately implies the following: 
 
 \begin{proposition}\label{p.seporbit} 
 Let $A \subset B$ be a connected maximal subalgebra of separable type, where $B = kQ$ and $Q$ is a type $\mathbb{A}$ Dynkin quiver. Then $\Iso (A,B) = \Aut_k(B)\cdot A$.
 \end{proposition} 
 
 \noindent Finally, Propositions \ref{p.splitorbit} and \ref{p.seporbit} combine to yield Theorem \ref{t.main}. 
 
We end with an example to show that the connectedness hypothesis is necessary:
 
 \begin{example} 
 We note that the conclusion of Theorem \ref{t.main} is false if $A$ has a disconnected Ext quiver. For instance, suppose that $Q$ is the type $\mathbb{A}_4$ Dynkin quiver  
 
 \[
\vcenter{\hbox{ 
 \begin{tikzpicture}[point/.style={shape=circle,fill=black,scale=.3pt,outer sep=3pt},>=latex, baseline=-3,scale=2]  
\node[point,label={below:$v_{-2}$}] (1) at (0,0) {} ; 
\node[point,label={below:$v_{-1}$}] (2) at (1,0) {} ; 
\node[point,label={below:$v_1$}] (3) at (2,0) {} ;  
\node[point,label={below:$v_2$}] (4) at (3,0) {} ;
\path[->] (1) edge  (2)  
(3) edge (2) 
(3) edge (4) ;
\end{tikzpicture}
}}.
\]
Then $\Aut_k(B) = \Inn (B)$ and $\Iso (A_{-2}, B) = \Inn(B)\cdot A_{-2} \cup \Inn(B)\cdot A_1$. Of course, $A_{-2}$ and $A_1$ lie in different $\Inn(B)$-orbits, since their Jacobson radicals are distinct.
 \end{example}


\end{document}